\documentclass{amsart}
\usepackage{mathrsfs}
\usepackage{amsfonts}
\usepackage{mathrsfs}
\usepackage{amssymb}
\usepackage{amsmath,amssymb}

\newtheorem{theorem}{Theorem}[section]
\newtheorem{proposition}[theorem]{Propisition}
\newtheorem{lemma}[theorem]{Lemma}

\theoremstyle{definition}
\newtheorem{definition}[theorem]{Definition}

\theoremstyle{remark}
\newtheorem{remark}[theorem]{Remark}

\numberwithin{equation}{section}

\begin{document}

\title[Well-posedness for fractional  Navier-Stokes equations]
{Well-posedness for fractional Navier-Stokes equations in critical
spaces close to
$\dot{B}^{-(2\beta-1)}_{\infty,\infty}(\mathbb{R}^{n})$}

\author{Zhichun Zhai}
\address{Department of Mathematics and Statistics, Memorial University of Newfoundland, St. John's, NL A1C 5S7, Canada}
\curraddr{}
 \email{a64zz@mun.ca}
\thanks{Project supported in part  by Natural Science and
Engineering Research Council of Canada.}

\subjclass[2000]{Primary 35Q30; 76D03; 42B35; 46E30}

\date{}

\keywords{Navier-Stokes equations;  $BMO^{-\zeta}(\mathbb{R}^{n});$
$Q_{\alpha}^{\beta}(\mathbb{R}^{n});$ Besov spaces}

\begin{abstract}
 In this paper,
we prove  the well-posedness for the fractional Navier-Stokes
equations in  critical spaces $G^{-(2\beta-1)}_{n}(\mathbb{R}^{n})$
and $BMO^{-(2\beta-1)}(\mathbb{R}^{n}).$ Both of them are close to
 the largest critical space $\dot{B}^{-(2\beta-1)}_{\infty,\infty}(\mathbb{R}^{n}).$ In
$G^{-(2\beta-1)}_{n}(\mathbb{R}^{n}),$  we establish the
well-posedness based on  a priori estimates for  the fractional
Navier-Stokes equations in Besov spaces. To obtain  the
well-posedness in $BMO^{-(2\beta-1)}(\mathbb{R}^{n}),$   we find a
relationship between $Q_{\alpha;\infty}^{\beta,-1}(\mathbb{R}^{n})$
and $BMO(\mathbb{R}^{n})$ by giving an equivalent characterization
of $BMO^{-\zeta}(\mathbb{R}^{n}).$
\end{abstract} \maketitle


 \vspace{0.1in}

 \section{Introduction}
In this paper, we study  the well-posedness  of mild solutions to
the fractional
 Navier-Stokes equations on the half-space
$\mathbb{R}^{1+n}_{+}=(0,\infty)\times\mathbb{R}^{n},$ $n\geq 2:$
\begin{equation}\label{eq1e}
\left\{\begin{array} {l@{\quad \quad}l}
\partial_{t}u+(-\triangle)^{\beta}u+(u\cdot \nabla)u-\nabla p=0,
& \hbox{in}\ \mathbb{R}^{1+n}_{+};\\
\nabla \cdot u=0, & \hbox{in}\ \mathbb{R}^{1+n}_{+};\\
u|_{t=0}=a, &\hbox{in}\ \mathbb{R}^{n} \end{array} \right.
\end{equation}
with $\beta\in(1/2,1).$ The mild  solution to equations (\ref{eq1e})
is the fixed point of operator
 $$(Tu)(t,x)=e^{-t(-\triangle)^{\beta}}a(x)-\int_{0}^{t}e^{-(t-s)(-\triangle)^{\beta}}P\nabla(u\otimes u)(s,x)ds.$$
 Here
$$e^{-t(-\triangle)^{\beta}}f(x)
:=K_{t}^\beta(x)\ast f(x)\ \text{ with }
 \widehat{K_{t}^{\beta}}(\xi)=e^{-t|\xi|^{2\beta}}$$
 and $P$ is the Helmboltz-Weyl projection:
 $$P=\{P_{j,k}\}_{j,k=1,\cdots,n}=\{\delta_{j,k}+R_{j}R_{k}\}_{j,k=1,\cdots,n}$$
 with the Kronecker symbol $\delta_{j,k}$   and the Riesz transform
 $R_{j}=\partial_{j}(-\triangle)^{-1/2}.$

Note that the following scaling
 \begin{equation}\label{eq2e}
 u_{\lambda}(t,x)=\lambda^{2\beta-1}u(\lambda^{2\beta}t,\lambda x),
\ \
 p_{\lambda}(t,x)= \lambda^{4\beta-2}p(\lambda^{2\beta}t,\lambda x),\ \
 a_{\lambda}(x)= \lambda^{2\beta-1} a(\lambda x)
\end{equation}
 is important for  equations (\ref{eq1e}). This leads us to study equations (\ref{eq1e})
 in critical function spaces   which  are invariant
under the scaling $f(x)\longrightarrow \lambda^{2\beta-1}f(\lambda
x).$

When $\beta=1,$ equations (\ref{eq1e}) become the classical
Navier-Stokes equations.   The existence of mild solutions  has been
established locally in time and global for small initial data in
various critical  spaces. Especially,   Koch and Tataru in \cite{H.
Koch D. Tataru} proved the well-posedness of classical Navier-Stokes
equations in
   the space
$BMO^{-1}(\mathbb{R}^{n})=\nabla\cdot(BMO(\mathbb{R}^{n}))^{n}.$
Xiao  in  \cite{J. Xiao 1} generalized  the results of Koch and
Tataru \cite{H. Koch D. Tataru} to
$Q_{\alpha;\infty}^{-1}(\mathbb{R}^{n})$ for $\alpha\in (0,1).$ Chen
and Xin in \cite{Chen Xin} studied the classical Navier-Stokes
equations in several critical spaces.
  See,  Kato \cite{T Kato}, Cannone
\cite{M Cannone}, Giga and Miyakawa \cite{Y Giga T Miyakawa},
  Bourgain and
Pavlovi\'{c}  \cite{Bourgain Pavlovic}
 and the references therein for more history and recent development.

For general case,  Lions
 \cite{J L Lions}  proved the
global  existence of  the classical solutions to equations
(\ref{eq1e}) when $\beta\geq \frac{5}{4}$ in dimensional $3.$ Wu in
\cite{J Wu 1} obtained similar result   for  $\beta\geq
\frac{1}{2}+\frac{n}{4}$ in dimension $n.$  For the important case
$\beta< \frac{1}{2}+\frac{n}{4},$ Wu in \cite{J Wu 2}-\cite{J Wu 3}
considered the existence of solution  to equations (\ref{eq1e}) in
$\dot{B}^{1+\frac{n}{p}-2\beta}_{p,q}(\mathbb{R}^{n}).$ In Li and
Zhai \cite{Pengtao Li Zhicun Zhai}-\cite{Li Zhai}, inspired by Koch
and Tataru \cite{H. Koch D. Tataru} and Xiao \cite{J. Xiao 1},
 they studied  equations (\ref{eq1e})  in critical space
   $Q_{\alpha;\infty}^{\beta,-1}(\mathbb{R}^{n})=\nabla\cdot(Q_{\alpha}^{\beta}(\mathbb{R}^{n}))^{n}$
for $\beta\in (1/2,1)$ and $\alpha\in (0,\beta).$
    Here  $Q_{\alpha}^{\beta}(\mathbb{R}^{n})$  for $\alpha\in (-\infty,\beta)$ is
   the set of all measurable functions with
 $$\sup_{I}(l(I))^{2(\alpha+\beta-1)-n}\int_{I}\int_{I}\frac{|f(x)-f(y)|^{2}}{|x-y|^{n+2(\alpha-\beta+1)}}dxdy<\infty$$
 where  the supremum is taken over all cubes $I$ with the edge length $l(I)$
 and the edges parallel to the coordinate axes in $\mathbb{R}^{n}.$ $Q_{\alpha}^{\beta}(\mathbb{R}^{n})$ is a
 generalization of $Q_{\alpha}(\mathbb{R}^{n})$ studied by Essen,
Janson, Peng and Xiao  \cite{M. Essen S. Janson L. Peng J. Xiao},
Xiao \cite{J. Xiao}, Dafni and Xiao \cite{G. Dafni J. Xiao}-\cite{G.
Dafni J. Xiao 1}.
 Meanwhile, Li and
Zhai \cite{Pengtao Li Zhicun Zhai}  proved that   Besov space
$\dot{B}^{1-2\beta}_{\infty,\infty}(\mathbb{R}^{n})$ for $\beta\in
(1/2,1)$ is the biggest one among the critical spaces of equations
(\ref{eq1e}).

In this paper, we accomplish two major  goals. First, we prove the
well-posedness for equations (\ref{eq1e}) in  spaces
 $G_{n}^{-(2\beta-1)}(\mathbb{R}^{n})$ for $\beta\in (1/2,1).$  Here, for
$s>0,$
$$G_{p}^{-s}(\mathbb{R}^{n})=\left\{f\in \mathcal {S}'(\mathbb{R}^{n}): |f|\in \mathcal {S}'(\mathbb{R}^{n}),
\|f\|_{G_{p}^{-s}(\mathbb{R}^{n})}=\sup_{t>0}t^{\frac{sn}{2p\beta}}\|e^{-t(-\triangle)^{\beta}}|f|\|_{L^{\infty}(\mathbb{R}^{n})}<\infty
\right\}.$$ Second, to obtain the well-posedness  in
$BMO^{-(2\beta-1)}(\mathbb{R}^{n})$ for $\beta\in (1/2,1),$  we find
a relation between $Q_{\alpha,\infty}^{\beta,-1}(\mathbb{R}^{n})$
and $BMO(\mathbb{R}^{n}):$
\begin{equation}\label{relation}
Q_{\alpha,\infty}^{\beta,-1}(\mathbb{R}^{n})=(-\triangle)^{\frac{2\beta-1}{2}}BMO(\mathbb{R}^{n})=BMO^{-(2\beta-1)}(\mathbb{R}^{n})
\end{equation}
for $\alpha=1-\beta$ and $\beta\in (1/2,1),$
 by  giving an
equivalent characterization of
 $BMO^{-\zeta}(\mathbb{R}^{n}).$
 Our
well-posedness results   extend   that of Chen and Xin \cite{Chen
Xin},   Koch and Tataru  \cite{H. Koch D. Tataru}.  The relation
(\ref{relation}) between
$Q_{\alpha,\infty}^{\beta,-1}(\mathbb{R}^{n})$ for $\beta\in
(1/2,1)$ and $BMO(\mathbb{R}^{n})$ gives us a clear link between
$Q_{\alpha}^{\beta}(\mathbb{R}^{n})$ and $BMO(\mathbb{R}^{n}).$ When
$\alpha\neq1-\beta,$ an interesting problem is whether or not  there
is a similar link between $Q_{\alpha}^{\beta}(\mathbb{R}^{n})$ and
$BMO(\mathbb{R}^{n}).$

 The space
$BMO^{-\zeta}(\mathbb{R}^{n})$ was introduced by Zhou and Gala in
\cite{Zhou Gala 1} by using heat semigroup $e^{t\triangle}.$ In the
following, we  define $BMO^{-\zeta}(\mathbb{R}^{n})$ by
$e^{-t(-\triangle)^{\beta}}$ for $\beta\in(1/2,1).$ This is
motivated
 by  the following well-known facts.

 For  a  $C^{\infty}$
real-valued function on $\mathbb{R}^{n}$  satisfying  the
properties:
\begin{equation}\label{prp psi}
\phi_{j}\in L^{1}(\mathbb{R}^{n}),\ |\phi_{j}(x)|\lesssim
(1+|x|)^{-(n+1)},\
 \int_{\mathbb{R}^{n}}\phi_{j}(x)dx=0\ \ \hbox{and}\
\ (\phi_{j})_{t}(x)=t^{-n}\phi_{j}\left(\frac{x}{t}\right),
\end{equation}
 \begin{equation}\label{eq cha q a b}
 f\in
BMO(\mathbb{R}^{n})\Longleftrightarrow \sup_{x\in \mathbb{R}^{n},
r\in(0,\infty)}r^{-n}\int_{0}^{r}\int_{|y-x|<r}|f\ast
\phi_{t}(y)|^{2}t^{-1}dtdy<\infty.
\end{equation}
Here $A\lesssim B$ means $A\leq CB$ with $C>0.$
 Thus $BMO(\mathbb{R}^{n})$  can be defined equivalently as
\begin{equation}\label{semi}
\|f\|^{2}_{BMO(\mathbb{R}^{n})}=\sup_{x\in\mathbb{R}^{n}, r\in
(0,\infty)}r^{-n} \int_{0}^{r^{2\beta}}\int_{|y-x|<r}|\nabla
e^{-t(-\triangle)^{\beta}}
f(y)|^{2}t^{\frac{1-\beta}{\beta}}dtdy<\infty.
\end{equation}

Then, (\ref{semi}) leads us to  introduce
$BMO^{-\zeta}(\mathbb{R}^{n})$ as follows.
\begin{definition}
For $\beta\in (1/2,1),$ $0\leq \zeta\leq n/2,$ define
$BMO^{-\zeta}(\mathbb{R}^{n})$ as the set of all measurable
functions $f$ with
$$\|f\|^{2}_{BMO^{-\zeta}(\mathbb{R}^{n})}=\sup_{x\in\mathbb{R}^{n},
r\in (0,\infty)}r^{-n}
\int_{0}^{r^{2\beta}}\int_{|y-x|<r}t^{\frac{\zeta-\beta}{\beta}}|
e^{-t(-\triangle)^{\beta}} f(y)|^{2}dtdy<\infty.$$
\end{definition}
\begin{remark}
 Obviously,  $BMO^{-\zeta}(\mathbb{R}^{n})$ is invariant under the
scaling $f(x)\longrightarrow \lambda^{\zeta}f(\lambda x).$ Note that
$Q_{\alpha;\infty}^{\beta,-1}(\mathbb{R}^{n})$ is invariant under
the scaling $f(x)\longrightarrow \lambda^{2\beta-1}f(\lambda x).$
Thus $BMO^{-\zeta}(\mathbb{R}^{n})$ will be more useful than
$Q_{\alpha;\infty}^{\beta,-1}(\mathbb{R}^{n}).$
\end{remark}

We  state our main results as follows. The first one is  a priori
 estimates in homogeneous Besov spaces for the fractional Navier-Stokes equations.

\begin{proposition}\label{lemma for priori esti}
Let $2-2\beta<w<2\beta,$ $1+n/p+w<4\beta,$ $2\leq n\leq p\leq
\infty,$ $1\leq q\leq \infty$ and
$$a\in (\mathcal {S}'(\mathbb{R}^{n}))^{n},
f(t)\in
(\dot{B}^{w-2\beta+\frac{n}{p}}_{p,\infty}(\mathbb{R}^{n}))^{n\times
n}.$$ Then the solution to the  integral equation
$$u(t)=e^{-t(-\triangle)^{\beta}}a+\int_{0}^{t}e^{-(t-s)(-\triangle)^{\beta}}P\nabla \cdot f(s)ds$$
satisfies the estimates
$$\|u(t)\|_{\dot{B}^{-1+\frac{n}{p}}_{p,q}(\mathbb{R}^{n})}
\lesssim
\|a\|_{\dot{B}^{-(2\beta-1)+\frac{n}{p}}_{p,q}(\mathbb{R}^{n})}
+\sup_{0<s<t}s^{\frac{w}{2\beta}+1-\frac{1}{\beta}}\|f(s)\|_{\dot{B}^{w-2\beta+\frac{n}{p}}_{p,\infty}(\mathbb{R}^{n})}$$
and
$$t^{\frac{w}{2\beta}}\|u(t)\|_{\dot{B}^{w-(2\beta-1)+\frac{n}{p}}_{p,\infty}(\mathbb{R}^{n})}\lesssim
\|a\|_{\dot{B}^{-(2\beta-1)+\frac{n}{p}}_{p,\infty}(\mathbb{R}^{n})}
+\sup_{0<s<t}s^{\frac{w}{2\beta}}\|f(s)\|_{\dot{B}^{w-2\beta+\frac{n}{p}}_{p,\infty}(\mathbb{R}^{n})}$$
provided the right-hand sides of the above inequalities are finite,
respectively.
\end{proposition}

Applying Proposition \ref{lemma for priori esti}, we obtain the
existence of solution to equations (\ref{eq1e}).

\begin{theorem}\label{biggest critical spaces}
Let $n\geq 2,$ $\beta\in(1/2,1),$ $\max\{2\beta-1,
2-2\beta\}<w<2\beta,$ $1+n/p+w<4\beta,$ $a\in
G_{n}^{-(2\beta-1)}(\mathbb{R}^{n}),$ $\nabla \cdot a=0.$ If
$\|a\|_{G_{n}^{-(2\beta-1)}(\mathbb{R}^{n})}$ is small enough, then
there is a unique   solution to (\ref{eq1e}) satisfying
$$\|u(t)\|_{G_{n}^{-(2\beta-1)}(\mathbb{R}^{n})}
+t^{\frac{1}{2\beta}}\|u(t)\|_{L^{\infty}(\mathbb{R}^{n})}
+t^{\frac{w}{2\beta}}\|u(t)\|_{\dot{B}^{w-(2\beta-1)}_{\infty,\infty}(\mathbb{R}^{n})}\lesssim
\|a\|_{G_{n}^{-(2\beta-1)}(\mathbb{R}^{n})}.$$
\end{theorem}

Similar to  Theorem \ref{biggest critical spaces}, we can prove the
existence of solutions to the fractional magnetohydrodynamics
equations
\begin{equation}\label{eq1ee}
\left\{\begin{array} {l@{\quad \quad}l}
\partial_{t}u+(-\triangle)^{\beta}u+u\cdot \nabla u+\nabla p-b\cdot\nabla b=0,
& \hbox{in}\ \mathbb{R}^{1+n}_{+};\\
\partial_{t}b+(-\triangle)^{\beta}b+b\cdot \nabla b-b\cdot\nabla u=0,
& \hbox{in}\ \mathbb{R}^{1+n}_{+};\\
\nabla \cdot u=\nabla\cdot b=0, & \hbox{in}\ \mathbb{R}^{1+n}_{+};\\
u|_{t=0}=u_{0},\  b|_{t=0}=b_{0}, &\hbox{in}\ \mathbb{R}^{n}.
\end{array} \right.
\end{equation}

We refer  the readers to Wu \cite{J Wu 1} and \cite{J Wu 4} and the
references therein for more information about this system.

\begin{theorem}\label{biggest critical spaces 2}
Let $n\geq 2,$  $\beta\in(1/2,1),$ $\max\{2\beta-1,
2-2\beta\}<w<2\beta,$ $1+n/p+w<4\beta,$ $(u_{0},b_{0})\in
G_{n}^{-(2\beta-1)}(\mathbb{R}^{n}),$ $\nabla \cdot u_{0}=0$ and
$\nabla \cdot b_{0}=0.$ If
$\|u_{0}\|_{G_{n}^{-(2\beta-1)}(\mathbb{R}^{n})}
+\|b_{0}\|_{G_{n}^{-(2\beta-1)}(\mathbb{R}^{n})}$ is small enough,
then there is a unique  mild solution to (\ref{eq1ee}) satisfying
$$\|u(t)\|_{G_{n}^{-(2\beta-1)}(\mathbb{R}^{n})}
+t^{\frac{1}{2\beta}}\|u(t)\|_{L^{\infty}(\mathbb{R}^{n})}
+t^{\frac{w}{2\beta}}\|u(t)\|_{\dot{B}^{w-(2\beta-1)}_{\infty,\infty}(\mathbb{R}^{n})}
\lesssim \|a\|_{G_{n}^{-(2\beta-1)}(\mathbb{R}^{n})},$$
$$\|b(t)\|_{G_{n}^{-(2\beta-1)}(\mathbb{R}^{n})}
+t^{\frac{1}{2\beta}}\|b(t)\|_{L^{\infty}(\mathbb{R}^{n})}
+t^{\frac{w}{2\beta}}\|b(t)\|_{\dot{B}^{w-(2\beta-1)}_{\infty,\infty}(\mathbb{R}^{n})}\lesssim
\|a\|_{G_{n}^{-(2\beta-1)}(\mathbb{R}^{n})}.$$
\end{theorem}

Using Proposition \ref{lemma for priori esti},  we get the
 existence of solutions to equations (\ref{eq1e}) in
$\dot{B}^{-(2\beta-1)+\frac{n}{p}}_{p,\infty}(\mathbb{R}^{n}).$

\begin{proposition}\label{proposition 1}
Let $n\geq 2,$ $\beta\in(1/2,1),$ $n\leq p<\infty,$
$\max\{2\beta-n/p,2-2\beta\}<w<2\beta$ and  $1+n/p+w<4\beta.$ Assume
that $a\in
(\dot{B}^{-(2\beta-1)+\frac{n}{p}}_{p,\infty}(\mathbb{R}^{n}))^{n}$
and $\nabla \cdot a=0.$ If
 $\|a\|_{\dot{B}^{-(2\beta-1)+\frac{n}{p}}_{p,\infty}(\mathbb{R}^{n})}$
is small enough, then there exists a unique solution to equations
(\ref{eq1e}) satisfying
$$\|u(t)\|_{\dot{B}^{-(2\beta-1)+\frac{n}{p}}_{p,\infty}(\mathbb{R}^{n})}
+t^{\frac{2\beta-1}{2\beta}}\|u(t)\|_{L^{\infty}(\mathbb{R}^{n})}
+t^{\frac{w}{2\beta}}\|u(t)\|_{\dot{B}^{w-(2\beta-1)+\frac{n}{p}}_{p,\infty}(\mathbb{R}^{n})}
\lesssim\|a\|_{\dot{B}^{-(2\beta-1)+\frac{n}{p}}_{p,\infty}(\mathbb{R}^{n})}.$$
\end{proposition}
\begin{remark}
In \cite{J Wu 3},Wu  established a result similar to Proposition
\ref{proposition 1} by using lower bounds for the integral involving
$(-\triangle)^{\beta}.$
\end{remark}

Now, we study the properties of $BMO^{-\zeta}(\mathbb{R}^{n}).$
\begin{proposition}\label{th 1}
$(BMO^{-\zeta}\ \hbox{and Besov spaces})$ Let $\beta\in
(\frac{1}{2},1).$ For any $f\in \mathcal{S}'(\mathbb{R}^{n})$ and
$t>0,$   we have
$$r^{
\zeta}\|e^{-r^{2\beta}(-\triangle)^\beta}f\|_{L^{\infty}}\lesssim\left(r^{-n}
 \int_{0}^{r^{2\beta}}\int_{|x-x_{0}|\leq
 r}s^{\zeta-1+\frac{1-\beta}{\beta}}|e^{-s(-\triangle)^\beta}f(x)|^{2}dxds\right)^{1/2},$$
that is, $BMO^{-\zeta}(\mathbb{R}^{n})\hookrightarrow
\dot{B}^{-\zeta}_{\infty,\infty}(\mathbb{R}^{n}).$
\end{proposition}

\begin{proposition}\label{fractioanl derivative of BMO}
A distribution $f$ belongs to $BMO^{-\zeta}(\mathbb{R}^{n})$ if and
only if there exists a distribution $g\in BMO(\mathbb{R}^{n})$ such
that $f=(-\triangle)^{\frac{\zeta}{2}}g.$
\end{proposition}

\begin{remark}
(i) Zhou and Gala established  results similar to Propositions
\ref{th 1}-\ref{fractioanl derivative of BMO} for
$BMO^{-\zeta}(\mathbb{R}^{n})$
 defined by heat semigroup $e^{t\triangle}.$ Thus,   $BMO^{-\zeta}(\mathbb{R}^{n})$ is independent of
$e^{-t(-\triangle)^{\beta}}$ for $\beta\in(1/2,1].$\\
(ii) It follows from the definition of
 $BMO^{-\zeta}(\mathbb{R}^{n})$ and
$Q_{\alpha;\infty}^{\beta,-1}(\mathbb{R}^{n})$ (see \cite{Pengtao Li
Zhicun Zhai})
 that when $\alpha=1-\beta,$
$Q_{\alpha;\infty}^{\beta,-1}(\mathbb{R}^{n})=BMO^{-\zeta}(\mathbb{R}^{n})$
for $\zeta=2\beta-1.$
 Thus,  we can
obtain the existence of mild solution to equations (\ref{eq1e}) with
initial data in $BMO^{-(2\beta-1)}(\mathbb{R}^{n})$ as follows.
\end{remark}
 We
need to define some notations.
\begin{definition}  \label{X space} Let $1/2<\beta<1.$ \\
(i)  A tempered distribution $f$ on $\mathbb{R}^{n}$ belongs to
$BMO_{T}^{-(2\beta-1)}(\mathbb{R}^{n})$ provided
 $$\|f\|_{BMO_{T}^{-(2\beta-1)}(\mathbb{R}^{n})}
 =\sup_{x\in\mathbb{R}^{n},r\in(0,T)}\left(r^{-n}
\int_{0}^{r^{2\beta}}\int_{|y-x|<r}| K_{t}^\beta\ast
f(y)|^{2}t^{\frac{\beta-1}{\beta}}dydt\right)^{1/2}<\infty;
$$
 (ii) A tempered distribution $f$ on $\mathbb{R}^{n}$ belongs to
$\overline{VBMO^{-(2\beta-1)}}(\mathbb{R}^{n})$ provided
$\lim\limits_{T\longrightarrow
0}\|f\|_{BMO^{-(2\beta-1)}_{T}(\mathbb{R}^{n})}=0;$\\
 (iii) A function $g$ on $\mathbb{R}^{1+n}_{+}$ belongs to the space
$X^{\beta}_{T}(\mathbb{R}^{n})$ provided
\begin{eqnarray*}
\|g\|_{X^{\beta}_{T}(\mathbb{R}^{n})}&=&\sup_{t\in(0,T)}t^{1-\frac{1}{2\beta}}\|g(t,\cdot)\|_{L^{\infty}(\mathbb{R}^{n})}\\
&+&\sup_{x\in
\mathbb{R}^{n},r^{2\beta}\in(0,T)}\left(r^{-n}\int_{0}^{r^{2\beta}}\int_{|y-x|<r}|g(t,y)|^{2}t^{\frac{\beta-1}{\beta}}dydt\right)^{1/2}<\infty.
\end{eqnarray*}
\end{definition}
\begin{proposition}\cite{Pengtao Li Zhicun Zhai} \label{th 3}
Let $n\geq 2,$  $1/2<\beta<1.$ Then\\
(i) The fractional Navier-Stokes system (\ref{eq1e}) has a unique
small global mild solution in $(X^{\beta}_{\infty})^{n}$ for all
initial data $a$ with $\nabla\cdot a=0$ and
$\|a\|_{(BMO_{\infty}^{-(2\beta-1)})^{n}}$ being small.\\
(ii) For any $T\in(0,\infty)$ there is an $\varepsilon>0$ such that
the fractional Navier-Stokes system (\ref{eq1e}) has a unique small
mild solution in $(X_{T}^{\beta})^{n}$ on $(0,T)\times
\mathbb{R}^{n}$ when the initial data $a$ satisfies $\nabla\cdot
a=0$ and $\|a\|_{(BMO_{T}^{-(2\beta-1)})^{n}}\leq \varepsilon.$ In
particular for all $a\in (\overline{VBMO^{-(2\beta-1)}})^{n}$ with
$\nabla\cdot a=0$ there exists a unique small local mild solution in
$(X_{T}^{\beta})^{n}$ on $(0,T)\times \mathbb{R}^{n}.$
 \end{proposition}
\begin{remark}
(i)   $G_{n}^{-(2\beta-1)}(\mathbb{R}^{n})$ and
$BMO^{-(2\beta-1)}(\mathbb{R}^{n})$ are different critical spaces
for equations (\ref{eq1e}) and no inclusion relation between them.\\
(ii) Proposition \ref{th 3} is an generalization   of Koch and
Tataru \cite[Theorem 2-3]{H. Koch D. Tataru} since
$BMO^{-(2\beta-1)}(\mathbb{R}^{n})=(-\triangle)^{-\frac{2\beta-1}{2}}BMO(\mathbb{R}^{n}).$\\
(iii) Similar to Proposition \ref{th 3}, we can consider the
well-posedness for dissipative quasi-geostrophic equations in
$BMO^{-(2\beta-1)}(\mathbb{R}^{2}).$
\end{remark}

The rest of this paper is organized as follows.  In Section 2, we
give the definition and some basic properties of Besov spaces. In
Section 3, we prove Proposition \ref{lemma for priori esti}. In
Section 4, we verify Theorem \ref{biggest critical spaces} based on
a prior estimates for fractional Navier-Stokes equations. In Section
5, we show Theorem \ref{biggest critical spaces 2} by the
contraction mapping principle. In Section 6,  we demonstrate
Proposition \ref{proposition 1} by applying the contraction mapping
principle and  a prior estimates for fractional Navier-Stokes
equations. In final two section, we establish   Propositions \ref{th
1} and \ref{fractioanl derivative of BMO}.

\section{Preliminary Lemmas}

In this section, we provide the definition  and several  properties
of the homogeneous Besov spaces.

 We recall  the definition  of homogeneous Besov spaces.
 For details, see  Berg and Lofstrom \cite{Berg
Lofstrom} and Triebel \cite{H. Triebel 2}-\cite{H. Triebel}. We
start with the fourier transform. The Fourier transform
$\widehat{f}$ of $f\in \mathcal {S}$ is defined as
$$\widehat{f}(\xi)=(2\pi)^{-n/2}\int_{\mathbb{R}^{n}}f(x)e^{- x\cdot \xi}dx.$$
Here $\mathcal{S}(\mathbb{R}^{n})$ denotes the Schwartz class of
rapidly decreasing smooth functions and
$\mathcal{S}'(\mathbb{R}^{n})$ is the space of tempered
distributions. The fractional power of the Laplacian can be defined
by the Fourier transform. For $\theta\in \mathbb{R},$
$$\widehat{(-\triangle)^{\theta/2}f}(\xi)=|\xi|^{\theta}\widehat{f}(\xi).$$
We will use $f^{\vee}$ to denote the inverse  Fourier transform of
$f.$ Then we introduce the Littlewood-Paley decomposition by means
of $\{\varphi_{j}\}_{j=-\infty}^{\infty}.$ Take a function $\phi\in
C_{0}^{\infty}$ with
$\hbox{supp}(\phi)=\{\xi\in\mathbb{R}^{n}:1/2<|\xi|\leq 2\}$ such
that $\sum_{j=-\infty}^{\infty}\phi(2^{-j}\xi)=1$ for all
$\xi\neq0.$ Then we define functions $\varphi_{j}(j=0, \pm 1,\pm
2,\cdots)$  as
$$\widehat{\varphi_{j}}(\xi)=\phi(2^{-j}\xi).$$
Let $\triangle_{j}f=\varphi_{j}\ast f,$ for $j=0,\pm 1,\pm2,\pm3,
\cdots.$ Then, for $s\in \mathbb{R}$ and $1\leq p,q\leq \infty,$ we
define
$$\|f\|_{\dot{B}^{s}_{p,q}(\mathbb{R}^{n})}
=\left(\sum_{j=-\infty}^{\infty}(2^{sj}\|\triangle_{j}
f\|_{L^{p}(\mathbb{R}^{n})})^{q}\right)^{1/q}, \ \ 1\leq q<\infty$$
$$\|f\|_{\dot{B}^{s}_{p,\infty}(\mathbb{R}^{n})}
=\sup_{-\infty<j<\infty}(2^{sj}\|\triangle_{j}
f\|_{L^{p}(\mathbb{R}^{n})}),\ \ q=\infty,$$ where
$L^{p}(\mathbb{R}^{n})$ means the usual Lebesgue space on
$\mathbb{R}^{n}$ with the norm $\|\cdot\|_{L^{p}(\mathbb{R}^{n})}.$
The homogeneous Bosev space $\dot{B}^{s}_{p,q}(\mathbb{R}^{n})$ is
defined by
$$\dot{B}^{s}_{p,q}(\mathbb{R}^{n})=\{f\in \mathcal {S}':
\|f\|_{\dot{B}^{s}_{p,q}(\mathbb{R}^{n})}<\infty\}.$$

We will use the following properties about homogeneous Besov space.

\begin{lemma} \label{le1}The following properties hold:\\
(i) If $1\leq q_{1}\leq q_{2}\leq \infty,$ $1\leq p\leq \infty$ and
$s\in \mathbb{R},$ then
$\dot{B}^{s}_{p,q_{1}}(\mathbb{R}^{n})\hookrightarrow
\dot{B}^{s}_{p,q_{2}}(\mathbb{R}^{n}).$\\
(ii) If $1\leq p_{1}\leq p_{2}\leq \infty,$ $1\leq q\leq \infty,$
$-\infty<s_{1}\leq s_{2}<\infty$ and
$s_{2}-\frac{n}{p_{2}}=s_{1}-\frac{n}{p_{2}},$ then
$$\dot{B}^{s_{2}}_{p_{2},q}(\mathbb{R}^{n})\hookrightarrow\dot{B}^{s_{1}}_{p_{1},q}(\mathbb{R}^{n}).$$
\\ (iii) If $\beta, s\in \mathbb{R},$ $1\leq p,q\leq \infty,$
then the operator $(-\triangle)^{\beta/2}$ is an isomorphism from
$\dot{B}^{s}_{p,q}(\mathbb{R}^{n})$ to
$\dot{B}^{s-\beta}_{p,q}(\mathbb{R}^{n}).$
\end{lemma}

\begin{lemma}\label{le 2}
Let $0<\theta<1,$ $1\leq p,q\leq p\leq \infty,$
$-\infty<s_{1}<s_{2}<\infty$ and $s=(1-\theta)s_{1}+\theta s_{2}.$
Then $$(\dot{B}^{s_{1}}_{p,\infty}(\mathbb{R}^{n}),
\dot{B}^{s_{2}}_{p,\infty}(\mathbb{R}^{n}))_{\theta,q}=\dot{B}^{s}_{p,q}(\mathbb{R}^{n})$$
for $s=s_{1}(1-\theta)+s_{2}\theta,$ where
$(\cdot,\cdot)_{\theta,q}$ means the real interpolation functor, see
Berg and  Lofstrom \cite{Berg Lofstrom}.
\end{lemma}

We will use the $L^{p}-L^{q}-$type estimates for
$e^{-t(-\triangle)^{\theta}}$ in homogeneous Besov spaces. For
$\theta=1,$ the $L^{p}-L^{p}$-estimates for $e^{t\triangle}$ in
Besov spaces were studied by Kozono, Ogawa and Taniuchi in
\cite{Hideo Kozono Takayoshi Ogawa Yasushi Taniuchi}. Zhai  in
\cite{Zhai} proved the general  case of $\theta>0.$
\begin{lemma}\label{lemma1}
Let $\theta>0$ and $\zeta\geq0.$  If  $s_{1}\leq s_{2},$ $1\leq
p_{1}\leq p_{2}\leq \infty$ and $1\leq q\leq \infty,$ then
\begin{equation}\label{acti 1}
\|e^{-t(-\triangle)^{\theta}}f\|_{\dot{B}^{s_{2}}_{p_{2},q}(\mathbb{R}^{n})}\lesssim
t^{-\frac{s_{2}-s_{1}}{2\theta}-\frac{n}{2\theta}\left(\frac{1}{p_{1}}-\frac{1}{p_{2}}\right)}
\|f\|_{\dot{B}^{s_{1}}_{p_{1},q}(\mathbb{R}^{n})}.
\end{equation}
\end{lemma}
The following equivalent characterization of homogeneous Besov
spaces will be useful.
\begin{lemma}(\cite{H. Triebel 2})\label{equiv norm in besov}
Let $0<s<1$ and $1\leq p\leq \infty,$ then in
$\dot{B}^{s}_{p,\infty}(\mathbb{R}^{n}),$ we have
$$\|f\|_{\dot{B}^{s}_{p,\infty}(\mathbb{R}^{n})}\equiv\sup_{y\neq 0}\frac{\|f(\cdot+y)-u(\cdot)\|_{L^{p}(\mathbb{R}^{n})}}{|y|^{s}}.$$
\end{lemma}

We need  a variant of Mikhlin theorem on Fourier multipliers.
\begin{lemma}\label{phi on Besov} (\cite{H. Triebel 2})
Let $-\infty<s<\infty$ and $\phi(x)$ be a complex-valued infinitely
differentiable function on $\mathbb{R}^{n}\backslash\{0\}$ so that
$$\sup_{j\leq k}\sup_{x\in\mathbb{R}^{n}}|x|^{j}|\nabla^{j}\phi(x)|<\infty$$
for a sufficiently large positive integer $k.$ Then
$$\|(\phi\widehat{u})^{\vee}\|_{\dot{B}^{s}_{p,q}(\mathbb{R}^{n})}
\lesssim\|u\|_{\dot{B}^{s}_{p,q}(\mathbb{R}^{n})}$$ for
$u\in\dot{B}^{s}_{p,q}(\mathbb{R}^{n})$ with $1\leq p,q\leq \infty.$
\end{lemma}

We need a useful lemma, see for example,  Grafakos \cite{Grafakos},
Frazier, Jawerth and Weiss \cite{Frazier Jawerth Weiss}.
\begin{lemma}\label{carleson measure BMO}
Let $f\in \mathcal {S}'(\mathbb{R}^{n}).$ Then the following
statements are equivalent:\\
 (i) $f\in BMO(\mathbb{R}^{n});$\\
(ii) for all $\phi\in\mathcal{S}'(\mathbb{R}^{n})$ satisfying:
$$\int_{\mathbb{R}^{n}}\phi(x)dx=0,\ \ \ \ \ \ \sup_{\xi\in\mathbb{R}^{n}}
\int_{0}^{\infty}|\widehat{\phi}(t\xi)|^{2}\frac{dtd\xi}{t}<\infty$$
and $ |\phi(x)|\lesssim\frac{1}{(1+|x|)^{n+c}}$  for some $c,$
 then
the measure $$d\mu(t,x)=|\phi_{t}\ast b(x)|^{2}\frac{dtdx}{t}$$ is a
Carleson measure on $\mathbb{R}^{1+n}_{+}.$
\end{lemma}

\begin{lemma}\label{lemma in Besov space} Let $2\beta-1<w<2\beta,$ $2\leq n\leq p\leq \infty,$ $1\leq q\leq
\infty,$ then we have
\begin{eqnarray*}
t^{\frac{2\beta-1}{2\beta}}\|u(t)\|_{L^{\infty}(\mathbb{R}^{n})}
+t^{\frac{w-(2\beta-1)}{2\beta}}\|u(t)\|_{\dot{B}^{w-2\beta+\frac{n}{p}}_{p,\infty}(\mathbb{R}^{n})}\\
\lesssim
\|u(t)\|_{\dot{B}^{-(2\beta-1)+\frac{n}{p}}_{p,\infty}(\mathbb{R}^{n})}
+t^{\frac{w}{2\beta}}\|u(t)\|_{\dot{B}^{w-(2\beta-1)+\frac{n}{p}}_{p,\infty}(\mathbb{R}^{n})}.\end{eqnarray*}
\end{lemma}
\begin{proof} It follows from Lemmas \ref{le1}-\ref{le 2} and \cite[Proposition 2.5.7]{H. Triebel
2} that
$$\dot{B}^{w-2\beta}_{p,\infty}(\mathbb{R}^{n})
=\left(\dot{B}^{-(2\beta-1)}_{p,\infty}(\mathbb{R}^{n}),
\dot{B}^{w-(2\beta-1)}_{p,\infty}(\mathbb{R}^{n})\right)_{\frac{w-1}{w},\infty}$$
and
\begin{eqnarray*}
\dot{B}^{0}_{\infty,\infty}(\mathbb{R}^{n})\supset
L^{\infty}(\mathbb{R}^{n})\supset\dot{B}^{0}_{\infty,1}(\mathbb{R}^{n})
=\left(\dot{B}^{-(2\beta-1)}_{\infty,\infty}(\mathbb{R}^{n}),
\dot{B}^{w-(2\beta-1)}_{\infty,\infty}(\mathbb{R}^{n})\right)_{\frac{2\beta-1}{w},1}
\end{eqnarray*}
which contains
$$\left(\dot{B}^{-(2\beta-1)+\frac{n}{p}}_{p,\infty}(\mathbb{R}^{n}),
\dot{B}^{w-(2\beta-1)+\frac{n}{p}}_{p,\infty}(\mathbb{R}^{n})\right)_{\frac{2\beta-1}{w},1}.
$$
Hence, we can get
\begin{eqnarray*}
&&\|u(t,t^{\frac{1}{2\beta}}\cdot)\|_{L^{\infty}(\mathbb{R}^{n})}
+\|u(t,t^{\frac{1}{2\beta}},\cdot)\|_{\dot{B}^{w-2\beta+\frac{n}{p}}_{p,\infty}(\mathbb{R}^{n})}\\
&\lesssim&\|u(t,t^{\frac{1}{2\beta}},\cdot)\|_{\dot{B}^{-(2\beta-1)+\frac{n}{p}}_{p,\infty}(\mathbb{R}^{n})}
+\|u(t,t^{\frac{1}{2\beta}}\cdot)\|_{\dot{B}^{w-(2\beta-1)+\frac{n}{p}}_{p,\infty}(\mathbb{R}^{n})}.
\end{eqnarray*}
By changing variables, we can find that
\begin{eqnarray*}
t^{\frac{2\beta-1}{2\beta}}\|u(t)\|_{L^{\infty}(\mathbb{R}^{n})}+
t^{\frac{w-(2\beta-1)}{2\beta}}\|u(t)\|_{\dot{B}^{w-2\beta+\frac{n}{p}}_{p,\infty}(\mathbb{R}^{n})}\\
\lesssim\|u(t)\|_{\dot{B}^{-(2\beta-1)+\frac{n}{p}}_{p,\infty}(\mathbb{R}^{n})}
+t^{\frac{w}{2\beta}}\|u(t)\|_{\dot{B}^{w-(2\beta-1)+\frac{n}{p}}_{p,\infty}(\mathbb{R}^{n})}.
\end{eqnarray*}
\end{proof}
\begin{lemma}\label{product}
For $\beta\in (1/2,1),$ $u,v\in (L^{\infty}(\mathbb{R}^{n}))^{n}\cap
(G_{n}^{-(2\beta-1)}(\mathbb{R}^{n}))^{n},$ then we have
$$\|e^{-t(-\triangle)^{\beta}}P\nabla \cdot(u\otimes v)\|_{G_{n}^{-(2\beta-1)}(\mathbb{R}^{n})}
\lesssim
t^{\frac{1-2\beta}{2\beta}}\|u\|_{L^{\infty}(\mathbb{R}^{n})}\|v\|_{G_{n}^{-(2\beta-1)}(\mathbb{R}^{n})}.$$
\end{lemma}
\begin{proof}
It is easy to see that for $\beta\in (1/2,1),$
$$\|\partial_{x_{i}}\partial_{x_{j}}\triangle^{-1}\partial_{x_{k}}K_{t}^{\beta}(x)\|_{L^{1}(\mathbb{R}^{n})}
\lesssim t^{-\frac{1}{2\beta}} (i,j,k =1,2,\cdots, n).$$ Since the
operation with respect to the convolution is commutative,  by
letting
$$K_{i,j,k,t}=(\delta_{ij}-\partial_{x_{i}}\partial_{x_{j}}\triangle^{-1})\partial_{x_{k}}K_{t}^{\beta}(x),$$
we have, for $s>0$
\begin{eqnarray*}
&&\|e^{-s(-\triangle)^{\beta}}|e^{-t(-\triangle)^{\beta}}P\nabla
\cdot(u\otimes v)|\|_{L^{\infty}(\mathbb{R}^{n})}\\
&\leq&\sum_{i=1}^{n}\sum_{j=1}^{n}\sum_{k=1}^{n}\|K_{s}^{\beta}\ast|K_{i,j,k,t}\ast(u_{k}v_{j})|\|_{L^{\infty}(\mathbb{R}^{n})}\\
&\leq&\sum_{i=1}^{n}\sum_{j=1}^{n}\sum_{k=1}^{n}\|K_{s}^{\beta}\ast|K_{i,j,k,t}|
\ast|u_{k}v_{j}|\|_{L^{\infty}(\mathbb{R}^{n})}\\
&\leq&\sum_{i=1}^{n}\sum_{j=1}^{n}\sum_{k=1}^{n}\||K_{i,j,k,t}\ast
K_{s}^{\beta}
\ast|u_{k}v_{j}|\|_{L^{\infty}(\mathbb{R}^{n})}\\
&\leq&\sum_{i=1}^{n}\sum_{j=1}^{n}\sum_{k=1}^{n}\|K_{i,j,k,t}\|_{L^{1}(\mathbb{R}^{n})}
\|K_{s}^{\beta}
\ast|u_{k}v_{j}|\|_{L^{\infty}(\mathbb{R}^{n})}\\
&\leq&\left(\|\nabla K_{t}^{\beta}\|_{L^{1}(\mathbb{R}^{n})}+
\sum_{i=1}^{n}\sum_{j=1}^{n}\sum_{k=1}^{n}\|\partial_{x_{i}}
\partial_{x_{j}}\triangle^{-1}\partial_{x_{k}}K_{t}^{\beta}\|_{L^{1}(\mathbb{R}^{n})}\right)
\|e^{-s(-\triangle)^{\beta}}|u_{k}v_{j}|\|_{L^{\infty}(\mathbb{R}^{n})}\\
&\lesssim&
t^{-\frac{1}{2\beta}}\|u\|_{L^{\infty}(\mathbb{R}^{n})}\|e^{-s(-\triangle)^{\beta}}|v|\|_{L^{\infty}(\mathbb{R}^{n})}.
\end{eqnarray*}
Thus, we get
\begin{eqnarray*}
&&\sup_{s>}s^{\frac{2\beta-1}{2\beta}}\|e^{-s(-\triangle)}
|e^{-t(-\triangle)^{\beta}}P\nabla \cdot(u\otimes
v)|\|_{L^{\infty}(\mathbb{R}^{n})}\\
&\lesssim&
t^{-\frac{1}{2\beta}}\|u\|_{L^{\infty}(\mathbb{R}^{n})}\sup_{s>0}s^{\frac{2\beta-1}{2\beta}}
\|e^{-s(-\triangle)^{\beta}}|v|\|_{L^{\infty}(\mathbb{R}^{n})}
\end{eqnarray*} and finishes the proof.
\end{proof}

\section{Proof of Proposition \ref{lemma for priori esti}}
 It follows from Lemma \ref{phi on Besov} that
\begin{equation}\label{est in priori esti}
\|Pv\|_{\dot{B}^{s}_{p,q}(\mathbb{R}^{n})}+\|\nabla(-\triangle)^{-1/2}v\|_{\dot{B}^{s}_{p,q}(\mathbb{R}^{n})}
\lesssim\|v\|_{\dot{B}^{s}_{p,q}(\mathbb{R}^{n})}.
\end{equation}
On the other hand, it is easy to see that for $k\geq 0,$
$$\|\nabla^{k}e^{-t(-\triangle)^{\beta}}v\|_{L^{p}(\mathbb{R}^{n})}
\lesssim t^{-\frac{k}{2\beta}}\|v\|_{L^{p}(\mathbb{R}^{n})}.$$ Then
(iii) of Lemma \ref{le1}  tells us
\begin{eqnarray*}
&&\|u(t)-e^{-t(-\triangle)^{\beta}}a\|_{\dot{B}^{w-(2\beta-1)+\frac{n}{p}}_{p,\infty}(\mathbb{R}^{n})}\\
&\lesssim&\sup_{\tau>0}
\tau^{\frac{2\beta-\left[w-(2\beta-1)+\frac{n}{p}\right]}{2\beta}}
\|\triangle^{\beta+1}
\int_{0}^{t}e^{-(t-s+\tau)(-\triangle)^{\beta}}\triangle^{-1}
P\nabla\cdot
f(s)ds\|_{L^{p}(\mathbb{R}^{n})}\\
&\lesssim&\sup_{\tau>0}
\tau^{\frac{2\beta-\left[w-(2\beta-1)+\frac{n}{p}\right]}{2\beta}}
\int_{0}^{t}(t+\tau-s)^{-\frac{2\beta}{2\beta}}\|\triangle
e^{-\frac{(t+\tau-s)}{2}(-\triangle)^{\beta}}\triangle^{-1}
P\nabla\cdot f(s)\|_{L^{p}(\mathbb{R}^{n})}ds\\
&\lesssim&\sup_{\tau>0}
\tau^{\frac{2\beta-\left[w-(2\beta-1)+\frac{n}{p}\right]}{2\beta}}
\int_{0}^{t}(t+\tau-s)^{\frac{-4\beta+[w-(2\beta-1)+\frac{n}{p}]}{2\beta}}
\|\triangle^{-1}
P\nabla\cdot f(s)\|_{\dot{B}^{w-(2\beta-1)+\frac{n}{p}}_{p,\infty}(\mathbb{R}^{n})}ds\\
&\lesssim&\sup_{\tau>0}
\tau^{\frac{2\beta-\left[w-(2\beta-1)+\frac{n}{p}\right]}{2\beta}}
\int_{0}^{t}(t+\tau-s)^{\frac{-4\beta+[w-(2\beta-1)+\frac{n}{p}]}{2\beta}}
 \| f(s)\|_{\dot{B}^{w-2\beta+\frac{n}{p}}_{p,\infty}(\mathbb{R}^{n})}ds\\
&\lesssim&\sup_{\tau>0}
\tau^{\frac{2\beta-\left[w-(2\beta-1)+\frac{n}{p}\right]}{2\beta}}
\left(\int_{0}^{t/2}+\int_{t/2}^{t}\right)
(t+\tau-s)^{\frac{-4\beta+[w-(2\beta-1)+\frac{n}{p}]}{2\beta}}s^{-\frac{w}{2\beta}}ds\\
&&\times\sup_{0<s<t}s^{\frac{w}{2\beta}}\| f(s)\|_{\dot{B}^{w-2\beta+\frac{n}{p}}_{p,\infty}(\mathbb{R}^{n})}\\
&\lesssim&\sup_{\tau>0}
\tau^{\frac{2\beta-\left[w-(2\beta-1)+\frac{n}{p}\right]}{2\beta}}
t^{-\frac{w}{2\beta}}\int_{0}^{t}(t+\tau-s)^{\frac{-4\beta+[w-(2\beta-1)+\frac{n}{p}]}{2\beta}}ds
\times\sup_{0<s<t}s^{-\frac{w}{2\beta}}\| f(s)\|_{\dot{B}^{w-2\beta+\frac{n}{p}}_{p,\infty}(\mathbb{R}^{n})}\\
&&+\sup_{\tau>0}
\tau^{\frac{2\beta-\left[w-(2\beta-1)+\frac{n}{p}\right]}{2\beta}}
(t+\tau)^{\frac{-4\beta+[w-(2\beta-1)+\frac{n}{p}]}{2\beta}}
\int_{0}^{t}s^{-\frac{w}{2\beta}}ds
\times\sup_{0<s<t}s^{\frac{w}{2\beta}}\| f(s)\|_{\dot{B}^{w-2\beta+\frac{n}{p}}_{p,\infty}(\mathbb{R}^{n})}\\
&\lesssim&\sup_{\tau>0}
\tau^{\frac{2\beta-\left[w-(2\beta-1)+\frac{n}{p}\right]}{2\beta}}
t^{-\frac{w}{2\beta}}\int_{0}^{t}(t+\tau-s)^{\frac{-4\beta+[w-(2\beta-1)+\frac{n}{p}]}{2\beta}}ds
\times\sup_{0<s<t}s^{-\frac{w}{2\beta}}\| f(s)\|_{\dot{B}^{w-2\beta+\frac{n}{p}}_{p,\infty}(\mathbb{R}^{n})}\\
&&+\sup_{\tau>0}
(t+\tau)^{\frac{2\beta-\left[w-(2\beta-1)+\frac{n}{p}\right]}{2\beta}+\frac{-4\beta+[w-(2\beta-1)+\frac{n}{p}]}{2\beta}}
\int_{0}^{t}s^{-\frac{w}{2\beta}}ds
\times\sup_{0<s<t}s^{\frac{w}{2\beta}}\| f(s)\|_{\dot{B}^{w-2\beta+\frac{n}{p}}_{p,\infty}(\mathbb{R}^{n})}\\
&\lesssim&\sup_{\tau>0}
\tau^{\frac{2\beta-\left[w-(2\beta-1)+\frac{n}{p}\right]}{2\beta}}
t^{-\frac{w}{2\beta}}\int_{0}^{t}(t+\tau-s)^{\frac{-4\beta+[w-(2\beta-1)+\frac{n}{p}]}{2\beta}}ds
\times\sup_{0<s<t}s^{-\frac{w}{2\beta}}\| f(s)\|_{\dot{B}^{w-2\beta+\frac{n}{p}}_{p,\infty}(\mathbb{R}^{n})}\\
&&+\sup_{\tau>0} (t+\tau)^{-1}t^{1-\frac{w}{2\beta}}
\times\sup_{0<s<t}s^{\frac{w}{2\beta}}\| f(s)\|_{\dot{B}^{w-2\beta+\frac{n}{p}}_{p,\infty}(\mathbb{R}^{n})}\\
&\lesssim&t^{-\frac{w}{2\beta}}\sup_{0<s<t}s^{\frac{w}{2\beta}}\|f(s)\|_{\dot{B}^{w-2\beta+\frac{n}{p}}_{p,\infty}(\mathbb{R}^{n})}.
\end{eqnarray*}
since $1+n/p+w<4\beta.$ Thus, by (iii) of Lemma \ref{le1} and Lemma
\ref{lemma1}, we have, for $2-2\beta<w<2\beta<2,$

\begin{eqnarray*}
&&\|u(t)-e^{-t(-\triangle)^{\beta}}a\|_{\dot{B}^{-(2\beta-1)+\frac{n}{p}}_{p,q}(\mathbb{R}^{n})}\\
&\lesssim&\int_{0}^{t}\|e^{-(t-s)(-\triangle)^{\beta}}\nabla\cdot
f(s)\|_{\dot{B}^{-(2\beta-1)+\frac{n}{p}}_{p,q}(\mathbb{R}^{n})}ds\\
&\lesssim&\int_{0}^{t}(t-s)^{-\frac{2-w}{2\beta}}
\|f(s)\|_{\dot{B}^{w-2\beta+\frac{n}{p}}_{p,q}(\mathbb{R}^{n})}ds\\
&\lesssim&\int_{0}^{t}(t-s)^{-\frac{2-w}{2\beta}}s^{-\frac{w}{2\beta}-(1-\frac{1}{\beta})}
\sup_{0<s,t}s^{\frac{w}{2\beta}+(1-\frac{1}{\beta})}\|f(s)\|_{\dot{B}^{w-2\beta+\frac{n}{p}}_{n,\infty}(\mathbb{R}^{n})}\\
&\lesssim&\sup_{0<s,t}s^{\frac{w}{2\beta}+1-\frac{1}{\beta}}\|f(s)\|_{\dot{B}^{w-2\beta+\frac{n}{p}}_{n,\infty}(\mathbb{R}^{n})}.
\end{eqnarray*}
Combining the previous estimates together, we get
$$ t^{\frac{w}{2\beta}}\|u(t)\|_{\dot{B}^{w-(2\beta-1)+\frac{n}{p}}_{p,\infty}(\mathbb{R}^{n})}\lesssim
t^{\frac{w}{2\beta}}\|e^{-t(-\triangle)^{\beta}}a\|_{\dot{B}^{w-(2\beta-1)+\frac{n}{p}}_{p,\infty}(\mathbb{R}^{n})}
+\sup_{0<s<t}s^{\frac{w}{2\beta}}\|f(s)\|_{\dot{B}^{w-2\beta+\frac{n}{p}}_{p,\infty}(\mathbb{R}^{n})},$$
and
$$\|u(t)\|_{\dot{B}^{-(2\beta-1)+\frac{n}{p}}_{p,q}(\mathbb{R}^{n})}
\lesssim
\|e^{-t(-\triangle)^{\beta}}a\|_{\dot{B}^{-(2\beta-1)+\frac{n}{p}}_{p,q}(\mathbb{R}^{n})}
+\sup_{0<s<t}s^{\frac{w}{2\beta}+1-\frac{1}{\beta}}\|f(s)\|_{\dot{B}^{w-2\beta+\frac{n}{p}}_{p,\infty}(\mathbb{R}^{n})}.
$$
Thus we can get  our estimates by applying Lemma \ref{lemma1} and
inequality (\ref{est in priori esti}).

\section{Proof of Theorem \ref{biggest critical spaces}}
Define $$X=\left\{u:[0,\infty)\longrightarrow
G_{n}^{-(2\beta-1)}(\mathbb{R}^{n})|\nabla \cdot u=0,
\|u\|_{X}<\infty\right\}$$ with
$$\|u\|_{X}=\sup_{t>0}\left(\|u(t)\|_{G_{n}^{-(2\beta-1)}(\mathbb{R}^{n})}+
t^{\frac{w}{2\beta}}\|u(t)\|_{\dot{B}^{w-(2\beta-1)}_{\infty,\infty}(\mathbb{R}^{n})}\right).$$
Set
$$T(u)(t)=e^{-t(-\triangle)^{\beta}}a-\int_{0}^{t}e^{-(t-s)(-\triangle)^{\beta}}P\nabla \cdot (u(s)\otimes v(s))ds.$$
We want to show that $T$ is a contraction  mapping from a ball of
$X$ to itself. The case of $p=\infty$ in Lemma \ref{lemma in Besov
space} implies that
\begin{eqnarray*}
t^{\frac{w-2\beta}{2\beta}}\|v(t)\|_{\dot{B}^{w-(2\beta-1)}_{\infty,\infty}(\mathbb{R}^{n})}
+t^{\frac{2\beta-1}{2\beta}}\|v(t)
\|_{L^{\infty}(\mathbb{R}^{n})}\lesssim\|v(t)\|_{\dot{B}^{-(2\beta-1)}_{\infty,\infty}(\mathbb{R}^{n})}
+t^{\frac{w}{2\beta}}\|v(t)\|_{\dot{B}^{w-(2\beta-1)}_{\infty,\infty}(\mathbb{R}^{n})}.
\end{eqnarray*}
Then, according to Proposition \ref{lemma for priori esti} and Lemma
\ref{equiv norm in besov}, we have
\begin{eqnarray*}
&&t^{\frac{w}{2\beta}}\|(Tu)(t)\|_{\dot{B}^{w-(2\beta-1)}_{\infty,\infty}(\mathbb{R}^{n})}\\
&\lesssim&\|a\|_{\dot{B}^{-(2\beta-1)}_{\infty,\infty}(\mathbb{R}^{n})}
+\sup_{0<s<t}s^{\frac{w}{2\beta}}\|u(s)\otimes
u(s)\|_{\dot{B}^{w-2\beta}_{\infty,\infty}(\mathbb{R}^{n})}\\
&\lesssim&\|a\|_{\dot{B}^{-(2\beta-1)}_{\infty,\infty}(\mathbb{R}^{n})}
+\sup_{0<s<t}s^{\frac{2\beta-1}{2\beta}}\|u(s)\otimes
v(s)\|_{\dot{B}^{-(2\beta-1)}_{\infty,\infty}(\mathbb{R}^{n})}
+\sup_{0<s<t}s^{\frac{w+2\beta-1}{2\beta}}\|u(s)\otimes u(s)\|_{\dot{B}^{w-(2\beta-1)}_{\infty,\infty}(\mathbb{R}^{n})}\\
&\lesssim&\|a\|_{\dot{B}^{-(2\beta-1)}_{\infty,\infty}(\mathbb{R}^{n})}
+\sup_{0<s<t}s^{\frac{2\beta-1}{2\beta}}\|u(s)\otimes
v(s)\|_{G_{n}^{-(2\beta-1)}(\mathbb{R}^{n})}
+\sup_{0<s<t}s^{\frac{w+2\beta-1}{2\beta}}\|u(s)\otimes u(s)\|_{\dot{B}^{w-(2\beta-1)}_{\infty,\infty}(\mathbb{R}^{n})}\\
&\lesssim&\|a\|_{\dot{B}^{-(2\beta-1)}_{\infty,\infty}(\mathbb{R}^{n})}
+\sup_{0<s<t}s^{\frac{2\beta-1}{2\beta}}\|u(s)\|_{L^{\infty(\mathbb{R}^{n})}}\|
v(s)\|_{G_{n}^{-(2\beta-1)}(\mathbb{R}^{n})}\\
&&+\sup_{0<s<t}s^{\frac{w+2\beta-1}{2\beta}}\|u(s)\|_{L^{\infty}(\mathbb{R}^{n})}
\|u(s)\|_{\dot{B}^{w-(2\beta-1)}_{\infty,\infty}(\mathbb{R}^{n})}\\
&\lesssim&\|a\|_{G_{n}^{-(2\beta-1)}(\mathbb{R}^{n})}+\|u\|^{2}_{X}.
\end{eqnarray*}
On the other hand, Lemma \ref{product} implies that
\begin{eqnarray*}
\|(Tu)(t)\|_{G^{w-(2\beta-1)}_{n}(\mathbb{R}^{n})}
&\lesssim&\sup_{s>0}s^{\frac{2\beta-1}{2\beta}}\|e^{-t(-\triangle)^{\beta}}e^{-s(-\triangle)^{beta}}|a|\|_{L^{\infty}(\mathbb{R}^{n})}\\
&&+
\int_{0}^{t}\|e^{-(t-s)(-\triangle)^{\beta}}P\nabla\cdot(u(s)\otimes
u(s))\|_{G^{-(2\beta-1)}_{n}(\mathbb{R}^{n})}ds\\
&\lesssim&\|a\|_{G^{-(2\beta-1)}_{n}(\mathbb{R}^{n})}+
\int_{0}^{t}(t-s)^{\frac{2\beta-1}{2\beta}}\|u(s)\|_{G^{-(2\beta-1)}_{n}}\|u(s)\|_{L^{\infty}(\mathbb{R}^{n})}\\
&\lesssim&\|a\|_{G^{-(2\beta-1)}_{n}(\mathbb{R}^{n})}+\|u\|^{2}_{X}.
\end{eqnarray*}
Hence, we get
$$\|Tu(t)\|_{X}\lesssim \|a\|_{G_{n}^{-(2\beta-1)}(\mathbb{R}^{n})}+\|u\|^{2}_{X} $$
and
$$\|Tu-Tv\|_{X}\lesssim(\|u\|_{X}+\|v\|_{X})\|u-v\|_{X}.$$
Therefore, the contraction mapping principle implies there exists a
 unique  solution to  equations (\ref{eq1e}) if $\|a\|_{G_{n}^{-(2\beta-1)}(\mathbb{R}^{n})}$
is small enough.

\section{Proof of Theorem \ref{biggest critical spaces 2}}
The solution $\{u,b\}$ to equations (\ref{eq1ee}) can be written as
$$u(t,x)=e^{-t(-\triangle)^{\beta}}u_{0}(x)-B(u,u)+B(u,u):=F_{1}(u,b),$$
$$b(t,x)=e^{-(-\triangle)^{\beta}}b_{0}(x)-B(u,b)+B(b,u):=F_{2}(u,b),$$
with
$$B(u,v)=\int_{0}^{t}e^{-(t-s)(-\triangle)^{\beta}}P\nabla\cdot(u\otimes v)(s)ds.$$
Define
$$Y=\left\{(u,b):(0,\infty)\longrightarrow G_{n}^{-(2\beta-1)(\mathbb{R}^{n})}|
\nabla \cdot u=\nabla \cdot b=0,\|(u,b)\|_{Y}<\infty\right\}$$ with
$$\|(u,b)\|_{Y}=\sup_{t>0}\left(\|(u,b)(t)\|_{G_{n}^{-(2\beta-1)}(\mathbb{R}^{n})}
+t^{\frac{w-1}{2\beta}}
\|(u,b)(t)\|_{\dot{B}^{w-(2\beta-1)}_{\infty,\infty}(\mathbb{R}^{n})}\right)<\infty,$$
$$\|(u,b)\|_{Y}=\|u\|_{Y}+\|b\|_{Y}.$$
We  want to show that $F_{1}$ and $F_{2}$ are contraction mappings
from a ball of $Y$ to itself. We rewrite the solution $(u,b)$ as
$$\left(\begin{array}{ccc}u\\b\end{array}\right)=\left(\begin{array}{ccc}F_{1}(u,b)\\F_{2}(u,b)\end{array}\right):=F(u,b).
$$
Then we have \begin{eqnarray*}
&&t^{\frac{w}{2\beta}}\|F_{1}(u,b)(t)\|_{\dot{B}^{w-(2\beta-1)}_{\infty,\infty}(\mathbb{R}^{n})}\\
&\lesssim&
\|u_{0}\|_{\dot{B}^{-(2\beta-1)}_{\infty,\infty}(\mathbb{R}^{n})}
+\sup_{0<s<t}s^{\frac{w}{2\beta}}\|(u\otimes u,b\otimes
b)(s)\|_{\dot{B}^{w-2\beta}_{\infty,\infty}(\mathbb{R}^{n})}\\
&\lesssim&\|u_{0}\|_{\dot{B}^{-(2\beta-1)}_{\infty,\infty}(\mathbb{R}^{n})}
+\sup_{0<s<t}s^{\frac{2\beta-1}{2\beta}}
\|(u\otimes u, b\otimes b)(s)\|_{\dot{B}^{-(2\beta-1)}_{\infty,\infty}(\mathbb{R}^{n})}\\
&&+\sup_{0<s<t}s^{\frac{w+2\beta-1}{2\beta}}\|(u\otimes u, b\otimes
b)(s)\|_{\dot{B}^{w-(2\beta-1)}_{\infty,\infty}(\mathbb{R}^{n})}\\
&\lesssim&\|u_{0}\|_{\dot{B}^{-(2\beta-1)}_{\infty,\infty}(\mathbb{R}^{n})}
+\sup_{0<s<t}s^{\frac{2\beta-1}{2\beta}}
\|(u\otimes u, b\otimes b)(s)\|_{G^{-(2\beta-1)}_{n}(\mathbb{R}^{n})}\\
&&+\sup_{0<s<t}s^{\frac{w+2\beta-1}{2\beta}}\|(u, b)(s)\|_{\dot{B}^{w-(2\beta-1)}_{\infty,\infty}(\mathbb{R}^{n})}\|(u,b)(s)\|_{L^{\infty}(\mathbb{R}^{n})}\\
&\lesssim&\|u_{0}\|_{G^{-(2\beta-1)}_{n}(\mathbb{R}^{n})}
+\sup_{0<s<t}s^{\frac{2\beta-1}{2\beta}}
\|(u,b)(s)\|_{L^{\infty}}\|(u, b)(s)\|_{G^{-(2\beta-1)}_{n}(\mathbb{R}^{n})}\\
&&+\sup_{0<s<t}s^{\frac{w+2\beta-1}{2\beta}}\|(u, b)(s)\|_{\dot{B}^{w-(2\beta-1)}_{\infty,\infty}(\mathbb{R}^{n})}\|(u,b)(s)\|_{L^{\infty}(\mathbb{R}^{n})}\\
&\lesssim&\|u_{0}\|_{G^{-(2\beta-1)}_{n}(\mathbb{R}^{n})}\\
&&+\sup_{0<s<t}\left[s^{\frac{2\beta-1}{2\beta}}
\|(u,b)(s)\|_{L^{\infty}}\left(\|(u,
b)(s)\|_{G^{-(2\beta-1)}_{n}(\mathbb{R}^{n})}
+s^{\frac{w}{2\beta}}\|(u, b)(s)\|_{\dot{B}^{w-(2\beta-1)}_{\infty,\infty}(\mathbb{R}^{n})}\right)\right]\\
&\lesssim&\|u_{0}\|_{G^{-(2\beta-1)}_{n}(\mathbb{R}^{n})}+\|(u,b)\|_{Y}^{2}.
\end{eqnarray*}
Similarly, we get
\begin{eqnarray*}
&&\|F_{1}(u,b)(t)\|_{G^{-(2\beta-1)}_{n}(\mathbb{R}^{n})}\\&\lesssim&
\sup_{s>0}s^{\frac{2\beta-1}{2\beta}}
\|e^{-t(-\triangle)^{\beta}}e^{-s(-\triangle)^{\beta}}|u_{0}|(s)\|_{L^{\infty}(\mathbb{R}^{n})}\\
&&+\int_{0}^{t}\|e^{-(t-s)(-\triangle)^{\beta}}P\nabla\cdot(u\otimes
u)(s)\|_{G^{-(2\beta-1)}_{n}(\mathbb{R}^{n})}ds\\
&&+\int_{0}^{t}\|e^{-(t-s)(-\triangle)^{\beta}}P\nabla\cdot(b\otimes
b)(s)\|_{G^{-(2\beta-1)}_{n}(\mathbb{R}^{n})}ds\\
&\lesssim&\|u_{0}\|_{G_{n}^{-(2\beta-1)}(\mathbb{R}^{n})}
+\int_{0}^{t}(t-s)^{\frac{2\beta-1}{2\beta}}\|(u,b)\|_{G_{n}^{-(2\beta-1)}(\mathbb{R}^{n})}
\|(u,b)(s)\|_{L^{\infty}(\mathbb{R}^{n})}ds\\
&\lesssim&\|u_{0}\|_{G_{n}^{-(2\beta-1)}(\mathbb{R}^{n})}+\|(u,b)\|_{Y}^{2}.
\end{eqnarray*}
Thus, we have
$$\|F_{1}(u,b)\|_{Y}\lesssim\|u_{0}\|_{G_{n}^{-(2\beta-1)}(\mathbb{R}^{n})}+\|(u,b)\|_{Y}^{2}$$
and
$$\|F_{1}(u,b)(t)-F_{1}(u',b')\|_{Y}\lesssim\|(u-u',b-b')\|_{Y}(\|(u,b)\|_{Y}+\|(u',b')\|_{Y}).$$
Similarly, we can prove that
$$\|F_{2}(u,b)\|_{Y}\lesssim\|u_{0}\|_{G_{n}^{-(2\beta-1)}(\mathbb{R}^{n})}+\|u\|_{Y}\|b\|_{Y}$$
and
$$\|F_{2}(u,b)(t)-F_{2}(u',b')\|_{Y}
\lesssim\|(u-u',b-b')\|_{Y}(\|(u,b)\|_{Y}+\|(u',b')\|_{Y}).$$
 These estimates imply that
$$\|F(u,b)-F(u',b')\|_{Y}\lesssim\|(u-u',b-b')\|_{Y}(\|(u,b)\|_{Y}+\|(u',b')\|_{Y}).$$
Therefore, the contraction mapping principle finishes the proof.
\section{Proof of Proposition \ref{proposition 1}}
Define
$$K=\left\{f\in L^{\infty}\left((0,\infty);\dot{B}^{-(2\beta-1)+\frac{n}{p}}_{p,\infty}(\mathbb{R}^{n})\right)
: \nabla \cdot f=0, \|f\|_{K}<\infty\right\}$$
 with
 $$\|f\|_{K}=\sup_{t>0}\left(\|u(t)\|_{\dot{B}^{-(2\beta-1)+\frac{n}{p}}_{p,\infty}(\mathbb{R}^{n})}
+t^{\frac{w}{2\beta}}\|u(t)\|_{\dot{B}^{w-(2\beta-1)+\frac{n}{p}}_{p,\infty}(\mathbb{R}^{n})}\right).$$
Let
$$Tu(t)=e^{-t(-\triangle)^{\beta}}a-\int_{0}^{t}e^{-(t-s)(-\triangle)^{\beta}}P\nabla\cdot(u(s)\otimes u(s))ds.$$ We
want to prove that $T$ is a contraction mapping from a ball of $K$
to itself. It follows from Lemmas \ref{equiv norm in besov} and
\ref{lemma in Besov space} that
\begin{eqnarray*}
&&\|Tu(t)\|_{\dot{B}^{-(2\beta-1)+\frac{n}{p}}_{p,\infty}(\mathbb{R}^{n})}
+t^{\frac{w}{2\beta}}\|Tu(t)\|_{\dot{B}^{w-(2\beta-1)+\frac{n}{p}}_{p,\infty}(\mathbb{R}^{n})}\\
&\lesssim&\|a\|_{\dot{B}^{-(2\beta-1)+\frac{n}{p}}_{p,\infty}(\mathbb{R}^{n})}
+t^{\frac{w}{2\beta}}\int_{0}^{t}\|e^{-(t-s)(-\triangle)^{\beta}}P\nabla\cdot(u(s)\otimes
u(s))\|_{\dot{B}^{w-(2\beta-1)+\frac{n}{p}}_{p,\infty}(\mathbb{R}^{n})}ds\\
&\lesssim&\|a\|_{\dot{B}^{-(2\beta-1)+\frac{n}{p}}_{p,\infty}(\mathbb{R}^{n})}
+\sup_{0<s<t}s^{\frac{w}{2\beta}}\|(u(s)\otimes
u(s))\|_{\dot{B}^{w-2\beta+\frac{n}{p}}_{p,\infty}(\mathbb{R}^{n})}\\
&\lesssim&\|a\|_{\dot{B}^{-(2\beta-1)+\frac{n}{p}}_{p,\infty}(\mathbb{R}^{n})}
+\sup_{0<s<t}s^{\frac{w}{2\beta}}\|u(s)\|_{L^{\infty}(\mathbb{R}^{n})}\|
u(s)\|_{\dot{B}^{w-2\beta+\frac{n}{p}}_{p,\infty}(\mathbb{R}^{n})}\\
&\lesssim&\|a\|_{\dot{B}^{-(2\beta-1)+\frac{n}{p}}_{p,\infty}(\mathbb{R}^{n})}
+\|u\|^{2}_{K}
\end{eqnarray*}
since $0<w-2\beta+\frac{n}{p}<1,$ $n\leq p<\infty.$ Similarly, we
get
$$\|Tu-Tv\|_{K}\lesssim(\|u\|_{K}+\|v\|_{K})\|u-v\|_{K}, \ \ \hbox{for}\ \ u,v\in
K.$$ Thus, these estimates imply that $T$ is a contraction mapping
for
$\|a\|_{\dot{B}^{-(2\beta-1)+\frac{n}{p}}_{p,\infty}}(\mathbb{R}^{n})$
small enough. Therefore, we can finish the proof by the contraction
mapping principle.

\section{Proof of Proposition \ref{th 3}}
\begin{proof}
We can write
$$e^{-t(-\triangle)^\beta}f=e^{-(t-u)(-\triangle)^\beta}e^{-u(-\triangle)^\beta}f$$
and%
$$e^{-t(-\triangle)^\beta}f(x)=\frac{2}{t}\int_{0}^{t/2}e^{-(t-u)(-\triangle)^{\beta}
}e^{-u(-\triangle)^\beta}fds.$$ According to the definition of
$e^{-t(-\triangle)^\beta},$ it is a convolution operator with a
positive Kernel
$K_{t}^\beta(x)=(2\pi)^{-n/2}\int_{\mathbb{R}^{n}}e^{ix\xi-t|\xi|^{2\beta}}d\xi$
satisfying%
$K_{t}^\beta(x)=\frac{1}{t^{n/2\beta}}K_{1}^\beta(\frac{x}{t^{1/2\beta}}).$
 Then, using H$\ddot{o}$lder's inequality, we obtain that
\begin{eqnarray*}
 &&|e^{-t(-\triangle)^\beta}f(x_{0})|=\left|\frac{2}{t}\int_{0}^{t/2}\int_{\mathbb{R}^{n}}K_{t-u}^\beta(x-x_{0})e^{-u(-\triangle)^\beta}f(x_{0})dxds\right|\\
&\lesssim&\frac{2}{t}\int_{0}^{t/2}\left(\int_{\mathbb{R}^{n}}K_{t-u}^\beta(x-x_{0})|e^{-u(-\triangle)^\beta}f(x_{0})|^{2}dx\right)^{1/2}
\left(\int_{\mathbb{R}^{n}}|K_{t-u}^\beta(x-x_{0})|dx\right)^{1/2}ds\\
&\lesssim&\frac{2}{t}\int_{0}^{t/2}\left(\int_{\mathbb{R}^{n}}K_{t-u}^\beta(x-x_{0})|e^{-u(-\triangle)^\beta}f(x_{0})|^{2}dx\right)^{1/2}ds\\
&\lesssim&
\left(\frac{2}{t^{\frac{\zeta-1}{\beta}}}\int_{0}^{t/2}\int_{\mathbb{R}^{n}}K_{t-u}^\beta(x-x_{0})u^{\frac{\zeta-1}{\beta}}|e^{-u(-\triangle)^\beta}f(x_{0})|^{2}dxds\right)^{1/2}.\\
\end{eqnarray*}
By Miao, Yuan and Zhang's \cite[Lemma 2.1]{C. Miao B. Yuan  B.
Zhang} , we have
$$K_{t-u}^\beta(x-x_{0})\lesssim\frac{1}{(t-u)^{n/2\beta}}\frac{1}{\left(1+\frac{|x-x_{0}|}{(t-u)^{1/2\beta}}\right)^{n+2\beta}}.$$%
Thus
\begin{eqnarray*}
I&=&\int_{\mathbb{R}^{n}}K_{t-u}^\beta(x-x_{0})u^{\frac{\zeta-1}{\beta}}|e^{-u(-\triangle)^\beta}f(x)|^{2}dx\\
&\lesssim&\int_{\mathbb{R}^{n}}\frac{1}{(t-u)^{n/2\beta}}\frac{1}{\left(1+\frac{|x-x_{0}|}{(t-u)^{1/2\beta}}\right)^{n+2\beta}}
u^{\frac{\zeta-1}{\beta}}|e^{-u(-\triangle)^\beta}f(x)|^{2}dx\\
&\lesssim&\int_{\frac{x-x_{0}}{(t-u)^{\frac{1}{2\beta}}}\in k+[0,1]^{n}}\frac{1}{(t-u)^{n/2\beta}}
\frac{1}{\left(1+\frac{|x-x_{0}|}{(t-u)^{1/2\beta}}\right)^{n+2\beta}}u^{\frac{\zeta-1}{\beta}}|e^{-u(-\triangle)^\beta}f(x)|^{2}dx.\\
\end{eqnarray*}
Since $0<u<\frac{t}{2}$ and $\frac{t}{2}<t-u<t,$ we can get
\begin{eqnarray*}
I\lesssim \frac{1}{t^{n/2\beta}}\int_{0}^{t/2}\int_{|x-x_{0}|\leq
t^{\frac{1}{2\beta}}}u^{\frac{\zeta-1}{\beta}}|e^{-u(-\triangle)^\beta}f(x)|^{2}dx.\end{eqnarray*}
 This gives
 \begin{eqnarray*}
 \|e^{-t^{2\beta}(-\triangle)^\beta}f\|_{L^{\infty}}&\lesssim&
  \left(\frac{2}{t^{\frac{\zeta-1}{\beta}}}\frac{1}{t^{n}}
 \int_{0}^{t^{2\beta}}\int_{|x-x_{0}|\leq
 t}s^{\frac{\zeta-1}{\beta}}|e^{-s(-\triangle)^\beta}f(x)|^{2}dxds\right)^{1/2}\\
 &\lesssim& \left(\frac{2}{t^{\frac{\zeta-1}{\beta}}}t^{-n}
 \int_{0}^{t^{2\beta}}\int_{|x-x_{0}|\leq
 t}s^{\frac{\zeta-1}{\beta}}|e^{-s(-\triangle)^\beta}f(x)|^{2}dxds\right)^{1/2},
 \end{eqnarray*}
 that is,
 $$t^{\zeta}\|e^{-t^{2\beta}(-\triangle)^\beta}f\|_{L^{\infty}}\lesssim\left(t^{-n}
 \int_{0}^{t^{2\beta}}\int_{|x-x_{0}|\leq
 t}s^{\frac{\zeta-1}{\beta}}|e^{-s(-\triangle)^\beta}f(x)|^{2}dxds\right)^{1/2}.$$
It follows from  Miao, Yuan and Zhang's \cite[Prorposition 2.1]{C.
Miao B. Yuan B. Zhang} that, for $s<0,$ $f\in
\dot{B}_{\infty,\infty}^{s}(\mathbb{R}^{n})$ if and only if
 $$\sup_{r>0}r^{-s/2\beta}\|e^{-r(-\triangle)^{\beta}}f\|_{L^{\infty}(\mathbb{R}^{n})}<\infty.$$
Thus, the previous estimate implies that
$BMO^{-\zeta}(\mathbb{R}^{n})
 \hookrightarrow \dot{B}^{-\zeta}_{\infty,\infty}(\mathbb{R}^{n}).$
\end{proof}

\section{Proof of Proposition \ref{fractioanl derivative of BMO}}
We need the following lemma which can be proved easily.
\begin{lemma}
For $\zeta\geq 0,$ $(-\triangle)^{\zeta/2}e^{-(-\triangle)^{\beta}}$
is a convolution operator with kernel $K^{\zeta,\beta}(x)\in
L^{1}(\mathbb{R}^{n}).$
\end{lemma}

We divide the proof into two parts. First,  we  prove that $f\in
BMO^{-\zeta}(\mathbb{R}^{n})$ under the  assumption  of  the
existence of a distribution $g\in BMO(\mathbb{R}^{n})$ with
$f=(-\triangle)^{\zeta/2}g.$ From this assumption, we have, for all
$s>0,$
$$s^{\zeta/\beta}|e^{-s(-\triangle)^{\beta}}(-\triangle)^{\frac{\zeta}{2}}g|^{2}
=|K_{\zeta,s^{\frac{1}{2\beta}}}\ast g|^{2}$$ with
$$K_{\zeta,s^{\frac{1}{2\beta}}}(x)=s^{-\frac{n}{2\beta}}K_{\zeta}(\frac{x}
{s^{\frac{1}{2\beta}}}).$$ Here $K_{\zeta}\in L^{1}(\mathbb{R}^{n})$
and
$$\widehat{K_{\zeta,s^{\frac{1}{2\beta}}}}(\xi)
=\widehat{K_{\zeta}}(s^{\frac{1}{2\beta}}\xi)
=s^{\frac{\zeta}{2\beta}}|\xi|^{\zeta}e^{-s|\xi|^{2\beta}}.$$
 Thus
 $$\int_{\mathbb{R}^{n}}K_{\zeta}(x)ds=0\
 \hbox{and}\ |K_{\zeta}(x)|\lesssim\frac{1}{(1+|x|)^{n+\zeta}}.$$
For more about the kernel of $e^{-t(-\triangle)^{\beta}},$ see Miao,
Yuan and Zhang \cite{C. Miao B. Yuan  B. Zhang}. Then we have
\begin{eqnarray*}
\sup_{\xi\in\mathbb{R}^{n}}\int_{0}^{\infty}|\widehat{K_{\zeta}}(t\xi)|^{2}
\frac{dt}{t}
&=&\sup_{|\xi|=1}\int_{0}^{\infty}|\widehat{K_{\zeta}}(t\xi)|^{2}\frac{dt}{t}\\
&=&\sup_{|\xi|=1}\int_{0}^{\infty}(t^{\zeta}e^{-t^{2\beta}})^{2}\frac{dt}{t}\\
&=&\int_{0}^{\infty}t^{2\zeta-1}e^{-2t^{2\beta}}dt\\
&=&\frac{2^{-\frac{\zeta}{\beta}}-1}{\beta}\int_{0}^{\infty}t^{\frac{\zeta}{\beta}-1}e^{-t}dt\\
&=&\frac{2^{-\frac{\zeta}{\beta}}-1}{\beta}\Gamma(\frac{\zeta}{\beta})<\infty,
\end{eqnarray*}
since $\frac{\zeta}{\beta}>0.$ So
$d\mu(x,s)=|K_{\zeta,s^{\frac{1}{2\beta}}}\ast g|^{2}\frac{dtdx}{s}$
is a Carleson measure and
$$\int\int_{0<s<t,|x-x_{0}|<t^{\frac{1}{2\beta}}}
|(K_{\zeta,s^{\frac{1}{2\beta}}}\ast g)(x)|^{2}\frac{dsdx}{s}\leq
C\|g\|^{2}_{BMO(\mathbb{R}^{n})}t^{\frac{n}{2\beta}}.$$
 That is
 $\|f\|_{BMO^{-\zeta}(\mathbb{R}^{n})}\leq C\|g\|_{BMO(\mathbb{R}^{n})}.$

Second, we prove the existence of $g\in BMO(\mathbb{R}^{n})$ with
$f=(-\triangle)^{\frac{\zeta}{2}}g$ when $f\in
BMO^{-\zeta}(\mathbb{R}^{n}).$ Proposition \ref{th 1} implies that
we can get
$$g=\sum_{j<0}g_{j}-g_{j}(0)+\sum_{j>0}g_{j}$$
with $g_{j}=\triangle_{j}g$ such that $f=(-\triangle)^{\zeta/2}g$
and  $g\in \dot{B}^{0}_{\infty,\infty}(\mathbb{R}^{n}).$
 In fact,
$$\widehat{g}(\xi)=\sum_{j<0}\widehat{g_{j}}(\xi)
-\widehat{g_{0}}(\xi)+\sum_{j>0}\widehat{g_{j}}(\xi),$$ and
\begin{eqnarray*}
|\xi|^{\zeta}\widehat{g}(\xi)&=&
\sum_{j<0}|\xi|^{\zeta}\widehat{g_{j}}(\xi)
-|\xi|^{\zeta}\widehat{g_{0}}(\xi)+\sum_{j>0}|\xi|^{\zeta}\widehat{g_{j}}(\xi)\\
&=&|\xi|^{\zeta}\sum_{j\in\mathbb{Z}}\widehat{g_{j}}(\xi)
=\sum_{j\in\mathbb{Z}}\widehat{\triangle_{j}}(f)(\xi)=\widehat{f}(\xi),
\end{eqnarray*}
according to the homogeneous Littlewood-Paley decomposition of $f.$
On the other hand, to see  $g\in
\dot{B}^{0}_{\infty,\infty}(\mathbb{R}^{n}),$ we have,
$$g_{j}=\triangle_{j}g=\triangle_{j}(-\triangle)^{\zeta/2}f,$$
and
\begin{eqnarray*}
\widehat{g_{j}}(\xi)&=&|\xi|^{-\zeta}\phi(2^{-j}\xi)\widehat{f}(\xi)\\
&=&2^{-j\zeta}|2^{-j}\xi|^{-\zeta}\phi(2^{-j}\xi)\widehat{f}(\xi)\\
&=&2^{-jr}h_{j}(\xi)|2^{-j}\xi|^{-\zeta}\phi(2^{-j}\xi)\widehat{f}(\xi).
\end{eqnarray*}
Here $h_{j}\in C_{0}^{\infty}(\mathbb{R}^{n})$ satisfying $h_{j}=1$
on $C_{j}$
 and $\hbox{supp}(h_{j})\subset 2C_{j}.$
Let $$g_{j}=2^{-j\zeta}\triangle_{j}f\ast
(h_{j}|2^{-j}\xi|^{-\zeta})^{\vee},$$
 where $(h_{j}|2^{-j}\xi|^{-\zeta})^{\vee}\in L^{\infty}(\mathbb{R}^{n}).$
It follows from $h_{j}|2^{-j}\xi|^{-\zeta}\in
l^{\infty}(\mathbb{Z})$ that
$\|\triangle_{j}g\|_{L^{\infty}(\mathbb{R}^{n})}\in
l^{\infty}(\mathbb{Z}).$

 We need to prove that $g\in
BMO(\mathbb{R}^{n}).$ In fact, let $\eta$ by
$$\widehat{\eta}(s^{\frac{1}{2\beta}}\xi)=|s^{\frac{1}{2\beta}}\xi|^{\zeta}e^{-s|\xi|^{2\beta}}.$$
So $$\widehat{\eta_{2^{\frac{1}{2\beta}}}\ast
g}(\xi)=\widehat{\eta}(s^{\frac{1}{2\beta}}\xi)\widehat{g}(\xi)
=|s^{\frac{1}{2\beta}}\xi|^{\zeta}e^{-s|\xi|^{\beta}}\widehat{g}(\xi)$$
and
$$\widehat{g}(\xi)=\sum_{j\in\mathbb{Z}}|\xi|^{-\zeta}\widehat{\triangle_{j}}(f)(\xi)
=|\xi|^{-\zeta}\widehat{f}(\xi).$$ This tells
 us  $$\widehat{\eta_{s^{\frac{1}{2\beta}}}\ast g}(\xi)
=s^{\frac{\zeta}{2\beta}}e^{-s|\xi|^{2\beta}}\widehat{f}(\xi).$$ So
$$\eta_{s^{\frac{1}{2\beta}}}\ast g(x)
=s^{\frac{\zeta}{2\beta}}e^{-s(-\triangle)^{\beta}}f(x).$$ It
follows from  $f\in BMO^{-\zeta}(\mathbb{R}^{n})$ and $\eta$
satisfying the assumptions of Lemma \ref{carleson measure BMO} that
\begin{eqnarray*}
&&|B(x_{0},t^{\frac{1}{2\beta}})|^{-1}
\int_{0}^{t}\int_{|x-x_{0}|<t^{\frac{1}{2\beta}}}
|\eta_{s^{\frac{1}{2\beta}}}\ast g|^{2}\frac{dsdx}{s}\\
&=&t^{\frac{n}{2\beta}}\int_{0<s<t}\int_{|x-x_{0}|<t^{\frac{1}{2\beta}}}
|s^{\frac{\zeta}{2\beta}}e^{-s(-\triangle)^{\beta}}f(x)|^{2}\frac{dsdx}{s}\\
&\lesssim& \sup_{t>0}\sup_{x_{0}\in\mathbb{R}^{n}}
\left(t^{-n}\int_{0}^{t^{2\beta}}\int_{|x-x_{0}|<t}s^{\frac{\zeta-1}{\beta}}|e^{-s(-\triangle)^{\beta}}f(x)|^{2}dsdx
\right)\\
&=&\|f\|^{2}_{BMO^{-\zeta}(\mathbb{R}^{n})}.
\end{eqnarray*}
The previous estimate and Lemma \ref{carleson measure BMO} imply
that  $g\in BMO(\mathbb{R}^{n}).$ This finishes the proof.

 \vspace{0.1in}
\noindent {\bf{Acknowledgements.}}
 The author would like  to thank  Professor  Jie Xiao for  all  helpful discussions and  kind encouragement.
At the same time, the author are  grateful to Professor Yong Zhou
for sending  article \cite{Zhou Gala 1}.

\end{document}